\documentclass[twoside,12pt,reqno]{amsart}
\usepackage{bbm}
\usepackage{amssymb}
\usepackage{fullpage}

\usepackage[pdftex,colorlinks,backref,pagebackref,hypertexnames=false,
           linkcolor=blue,urlcolor=blue,citecolor=blue]{hyperref}

\newtheorem{Theorem}{Theorem}[section]
\newtheorem{Lemma}[Theorem]{Lemma}
\newtheorem{Proposition}[Theorem]{Proposition}
\newtheorem{Corollary}[Theorem]{Corollary}

\theoremstyle{remark}
\newtheorem{Remark}[Theorem]{Remark}

\numberwithin{equation}{section}

\def\bar{\overline}
\def\sub{\subseteq}
\def\iso{\cong}
\def\isoto{\overset{\sim}{\longrightarrow}}
\def\longto{\longrightarrow}

\def\kk{{\mathbbm k}}

\def\ad{\operatorname{ad}}
\def\Ad{\operatorname{Ad}}
\def\Ext{\operatorname{Ext}}
\def\Lie{\operatorname{Lie}}
\def\rad{\operatorname{rad}}
\def\sspan{\operatorname{span}}

\def\C{{\mathbb C}}
\def\F{{\mathbb F}}
\def\Z{{\mathbb Z}}

\def\GL{\mathrm{GL}}
\def\SL{\mathrm{SL}}
\def\Sp{\mathrm{Sp}}
\def\SO{\mathrm{SO}}
\def\O{\mathrm{O}}

\def\b{\mathfrak b}
\def\g{\mathfrak g}
\def\h{\mathfrak h}
\def\l{\mathfrak l}
\def\s{\mathfrak s}

\def\gl{\mathfrak{gl}}
\def\sl{\mathfrak{sl}}
\def\sp{\mathfrak{sp}}
\def\so{\mathfrak{so}}

\def\cN{\mathcal N}
\def\cO{\mathcal O}
\def\cU{\mathcal U}
\def\cV{\mathcal V}

\title{\boldmath On $\sl_2$-triples for classical algebraic groups in positive characteristic}
\author{Simon M.~Goodwin and Rachel Pengelly}
\address{School of Mathematics,
University of Birmingham,
Birmingham, B15 2TT,
UK}
\email{s.m.goodwin@bham.ac.uk, rap943@bham.ac.uk}

\begin{document}

\begin{abstract}
Let $\kk$ be an algebraically closed field of characteristic $p > 2$, let $n \in \Z_{>0} $, and take $G$ to
be one of the classical algebraic groups $\GL_n(\kk)$, $\SL_n(\kk)$, $\Sp_n(\kk)$, $\O_n(\kk)$ or $\SO_n(\kk)$,
with $\g = \Lie G$.
We determine the maximal $G$-stable closed subvariety $\cV$ of the nilpotent
cone $\cN$ of $\g$ such that the $G$-orbits in
$\cV$ are in bijection with the $G$-orbits of $\sl_2$-triples $(e,h,f)$ with $e,f \in \cV$.  This
result determines to what extent the theorems of Jacobson--Morozov and Kostant
on $\sl_2$-triples hold for classical algebraic groups over an algebraically closed field of ``small'' odd
characteristic.
\end{abstract}

\maketitle

\section{Introduction}

The theory of $\sl_2$-triples has been a key tool in the study of nilpotent orbits
in the Lie algebra $\g$ of a reductive algebraic group $G$ over $\C$.  This theory is
underpinned by the famous theorems of Jacobson--Morozov, \cite[Theorem 2]{Morozov} and \cite[Theorem 3]{Jacobson},
and Kostant, \cite[Theorem 3.6]{Kostant}.  These theorems combine to show that any nilpotent
$e \in \g$ can be embedded in an $\sl_2$-triple $(e,h,f)$ in $\g$, which is unique up to
the action of the centralizer of $e$ in $G$.

Subsequently there was a great deal of interest in extending the theorems
of Jacobson--Morozov and Kostant to the setting of a reductive algebraic group $G$
over an algebraically closed field $\kk$ of characteristic $p > 2$.    Recently the problem of determining
under what restriction on $p$ the theorems of Jacobson--Morozov and Kostant
remain true was solved in the work of Stewart--Thomas. To state
this result we let $\g = \Lie G$
and let $\cN$ denote the nilpotent cone of $\g$.
In \cite[Theorem~1.1]{StewartThomas}
it is shown that there is a bijection
\begin{equation} \label{e:Kostantmap}
\{ G \text{-orbits of } \sl_2\text{-triples in }  \g  \} \longto \{ G\text{-orbits in }  \cN \}
\end{equation}
sending the $G$-orbit of an $\sl_2$-triple $(e,h,f)$ to the $G$-orbit of $e$ if and only if
$p > h$, where $h$ is the Coxeter number of $G$.
A notable earlier result of Pommerening established the existence of $\sl_2$-triples containing
a given $e \in \cN$  under the hypothesis that $p$ is good for $G$, see \cite[\S2.1]{Pommerening}.
We also mention that earlier bounds on $p$
for there to be a bijection as in \eqref{e:Kostantmap} were: $p > 4h-3$ given by
Springer--Steinberg in \cite[III.4.11]{SpringerSteinberg};
and $p > 3h-3$, given by Carter, using an argument of Spaltenstein, in \cite[Section~5.5]{Carter}.

It is a natural question to consider to what extent the map
from the set of $G$-orbits of $\sl_2$-triples $(e,h,f)$ with $e,f \in \cN$ to the set of $G$-orbits in $\cN$
sending the $G$-orbit of an $\sl_2$-triple $(e,h,f)$ to the $G$-orbit of $e$
fails to be a bijection in the case where $p \le h$.  A key problem is to determine the
maximal $G$-stable closed subvarieties $\cV$ of $\cN$ such that the restriction of this map to
\begin{equation} \label{e:KostantmapV}
\left\{ G\text{-orbits of} \ \sl_2\text{-triples } (e,h,f) \text{ with } e,f \in \cV  \right\}
\longto \{ G\text{-orbits in} \ \cV \}
\end{equation}
is a bijection.  In this paper we solve this problem for the cases where $G$ is one of
$\GL_n(\kk)$, $\SL_n(\kk)$, $\Sp_n(\kk)$, $\O_n(\kk)$ or $\SO_n(\kk)$.
We determine a maximal subvariety $\cV$ of $\cN$ such that
the map in \eqref{e:KostantmapV} is a bijection, and observe that $\cV$ is in fact the unique
maximal such subvariety.  We note that it is possible to give a general argument to prove that $\cV$ is unique,
see Remark~\ref{R:uniquemax}.

Our main theorem is stated in Theorem~\ref{T:classification}.  First we give some notation required for its statement.

Let $G$ be one of $\GL_n(\kk)$, $\SL_n(\kk)$, $\Sp_n(\kk)$, $\O_n(\kk)$ or $\SO_n(\kk)$, where we
assume $n$ is even in the $\Sp_n(\kk)$ case.
 The Jordan normal form of any element in $\cN$ corresponds to a partition $\lambda$ of $n$, we recall that this uniquely determines a $G$-orbit in $\cN$, except in the case where $G = \SO_n(\kk)$ and $\lambda$ is a very even partition,
for which there are two $G$-orbits.  We recall the parametrization of $G$-orbits in $\cN$ in more detail
in \S\ref{ss:nilpotentorbits}, and for now just mention that we use the notation
$x \sim \lambda$ to denote that the partition of $n$
given by the Jordan normal form of $x$ is $\lambda$.
The subvarieties of $\cN$ required for the statement of Theorem~\ref{T:classification} are
\begin{equation} \label{e:Np-1}
\cN^{p-1} := \{ x \in \cN \mid x^{p-1}=0 \},
\end{equation}
and
\begin{equation} \label{e:1Np}
{}^1\cN^p := \{ x \in \cN \mid \ x \sim (\lambda_1, \lambda_2, \dots, \lambda_m ), \ \lambda_1 \le p, \ \lambda_2 < p \}.
\end{equation}

\begin{Theorem} \label{T:classification}
Let $\kk$ be an algebraically
closed field of characteristic $p>2$.
Let $(G,\g,\cV)$ be one of the following.
\begin{enumerate}
  \item $G = \GL_n(\kk)$, $\g = \gl_n(\kk)$, $\cV = \cN^{p-1}$,
  \item $G = \SL_n(\kk)$, $\g = \sl_n(\kk)$, $\cV = \cN^{p-1}$,
  \item $G = \Sp_n(\kk)$, $\g = \sp_n(\kk)$, $\cV = \cN^{p-1}$,
  \item $G = \O_n(\kk)$, $\g = \so_n(\kk)$, $\cV = {}^1\cN^p$,
  \item $G = \SO_n(\kk)$, $\g = \so_n(\kk)$, $\cV = {}^1\cN^p$.
 \end{enumerate}
Then the map
\begin{equation} \label{e:Kostantmapthm}
\{ G\text{-orbits of} \ \sl_2\text{-triples } (e,h,f) \text{ with } e,f \in \cV  \}
\longto \{ G\text{-orbits in} \ \cV \}
\end{equation}
given by sending the $G$-orbit of an $\sl_2$-triple $(e,h,f)$ to the $G$-orbit of $e$ is a bijection.
Moreover, $\cV$ is the unique maximal $G$-stable closed subvariety of $\cN$ that satisfies this property.
\end{Theorem}

As we frequently consider $G$-stable closed subvarieties $\cV$ of $\cN$ such that
the map in \eqref{e:Kostantmapthm} is a bijection, we introduce a shorthand
for such varieties, and say that such a variety satisfies the {\em $\sl_2$-property}.
Then Theorem~\ref{T:classification} determines the unique maximal $G$-stable closed subvariety
$\cV$ of $\cN$ that satisfies the $\sl_2$-property.  Or in other words it states that for $e \in \cV$,
there exists a unique $\sl_2$-triple $(e,h,f)$ in $\g$ with $f \in \cV$
up to conjugacy by the centralizer of $e$ in $G$, and
moreover, $\cV$ is maximal with respect to this property.

We remark that the variety $\cV$ in Theorem~\ref{T:classification} is equal to
$\cN$ if and only if $p > h$.
With a small argument to show that, for $p > h$, any $\sl_2$-triple $(e,h,f)$ in $\g$
satisfies $e,f \in \cN$, the classical cases in \cite[Theorem~1.1]{StewartThomas} can be deduced.

We mention that Theorem~\ref{T:classification} can be used to deduce a result about $G$-completely
reducible subalgebras of $\g$ isomorphic to $\sl_2(\kk)$.  We do not go into details
here and just say that Theorem~\ref{T:classification} can be used to show that if a subalgebra $\s \iso \sl_2(\kk)$ of $\g$ has basis
$\{e,h,f\}$ such that $e,f \in \cV$, then $\s$ is $G$-completely reducible; and moreover, $\cV$ is
maximal for this property.  We refer to the introduction of \cite{StewartThomas} for some
discussion and references on $G$-complete reducibility.
We also mention a connection with the theory of
good $A_1$-subgroups as introduced in the work of Seitz in \cite{Seitz},
see also the theory of optimal $\SL_2$-homomorphisms in the work of McNinch in \cite{McNinch}.
Using Theorem~\ref{T:classification} it can be shown that if a subalgebra  $\s \iso \sl_2(\kk)$ of $\g$ has basis
$\{e,h,f\}$ such that $e,f \in \cV$, then $\s= \Lie S$ for some good $A_1$-subgroup $S$ of $G$; and moreover, $\cV$ is
maximal with respect to this property.

The theory of standard $\sl_2$-triples
as introduced by Premet--Stewart in \cite[\S2.4]{PremetStewart} is discussed in \S\ref{ss:standard}, and
is used for part of our proof of Theorem~\ref{T:classification}.
As a consequence of this theory it can be deduced that
the $\sl_2$-triples occurring in Theorem~\ref{T:classification}
are all standard $\sl_2$-triples.  Also it can be used to show
that there is a unique maximal $G$-stable closed subvariety of $\cN$
satisfying the $\sl_2$-property, see Remark~\ref{R:uniquemax}.

In future work we intend to investigate the problem of
determining maximal $\cV \sub \cN$ satisfying the $\sl_2$-property
in the case $G$ is of exceptional type.
We also mention that one may consider analogous problems for $\sl_2$-subalgebras
rather than $\sl_2$-triples.  For example in \cite[Theorem~1.3]{StewartThomas},
precise conditions are given for there to be a bijection between $G$-orbits in $\cN$
and $G$-orbits of $\sl_2$-subalgebras.

We give a summary of the structure of this paper, in which we
outline the ideas in our proof of Theorem~\ref{T:classification}.

In \S\ref{ss:sl2modules}, we recall some required representation theory of $\sl_2(\kk)$,
and give some notation.  Then in \S\ref{ss:nilpotentorbits} we recap the parametrization of nilpotent orbits for classical groups,
and also recall the closure order on these orbits.
In \S\ref{ss:standard}, we discuss the theory of standard $\sl_2$-triples,
from which we deduce that
if $\cV$ satisfies the $\sl_2$-property, then $\cV$ is a subvariety of the
restricted nilpotent cone $\cN^p := \{x \in \g \mid x^p = 0\}$, see
Proposition~\ref{P:cVsubcNp}.
In \S\ref{ss:SLp} we consider $\sl_2$-triples
for the case $G = \SL_p(\kk)$, and recall a known result that for $e \in \cN$
with $e \sim (p)$ there are multiple $\sl_2$-triples $(e,h,f)$ up to conjugacy by $\SL_p(\kk)$.
This can be used to show that if $\cV$ satisfies the $\sl_2$-property, then $\cV$ cannot contain
a nilpotent element that is regular in a Levi subalgebra of $\g$ whose
derived subalgebra is $\sl_p(\kk)$, see Proposition~\ref{P:SLp}.  It turns out that
Propositions~\ref{P:cVsubcNp} and \ref{P:SLp} are precisely what is needed to prove the
maximality of the subvarieties $\cV$ in Theorem~\ref{T:classification}.

In Section~\ref{S:glsl} we consider the cases where $G$ is
$\GL_n(\kk)$ or $\SL_n(\kk)$.  For the case $G = \GL_n(\kk)$, we consider the algebra
$A := U(\sl_2(\kk)) / \langle e^{p-1}, f^{p-1} \rangle$, and use that $A$ is a semisimple
algebra to prove that $\cV = \cN^{p-1}$ has the $\sl_2$-property in Corollary~\ref{C:gl}.
The semisimplicity of $A$ is given by a theorem of Jacobson, \cite[Theorem 1]{JacobsonTDSLA}, see also
\cite[Theorem~5.4.8]{Carter}, though we provide an alternative proof of this.  Our proof involves determining
a lower bound for $\dim (A/\rad A)$ and an upper bound for $\dim A$, which are equal, and from
this we can deduce that Jacobson radical of $A$ is zero.
We are then able to complete the proof of Theorem~\ref{T:classification}(a)
by noting that Proposition~\ref{P:SLp} implies maximality of $\cN^{p-1}$ satisfying the
$\sl_2$-property.  In \S\ref{ss:sl} we explain that
the case $G= \SL_n(\kk)$ in Theorem~\ref{T:classification}(b) follows quickly.

In Section~\ref{S:resultsforclassicals} we consider the cases where $G$ is one of
$\Sp_n(\kk)$, $\O_n(\kk)$ or $\SO_n(\kk)$.  A useful result for us is
Lemma~\ref{L:sl2orbitsarethesame}, which states that two $\sl_2$-triples in $\sp_n(\kk)$ (respectively
$\so_n(\kk)$) are conjugate by an element of $\GL_n(\kk)$  if and only if
they are conjugate by an element of $\Sp_n(\kk)$ (respectively $\O_n(\kk)$).
Using Lemma~\ref{L:sl2orbitsarethesame} and Theorem~\ref{T:classification}(a) we are able to
quickly show that for $G = \Sp_n(\kk)$ or $\O_n(\kk)$, we have that $\cN^{p-1}$ satisfies the $\sl_2$-property.
Then for $G = \Sp_n(\kk)$ we complete the proof of Theorem~\ref{T:classification}(c) by
using Propositions~\ref{P:cVsubcNp} and \ref{P:SLp} to deduce maximality.
We move on to deal with the case $G = \O_n(\kk)$ in \S\ref{ss:orthogonal}, where more
work is needed to consider
the case where $e \sim \lambda = (\lambda_1,\dots,\lambda_m)$ such that
$\lambda_1 = p$, $\lambda_2 < p$.  This requires
a detailed analysis of certain $\sl_2(\kk)$-modules, which is completed in the proof of Proposition~\ref{P:orthogonal}.
This proposition shows that $\cV = {}^1\cN^{p-1}$ does satisfy the $\sl_2$-property.
We then deduce maximality of $\cV$ similarly
to the previous cases to complete the proof of Theorem~\ref{T:classification}(d).
We are left to deduce Theorem~\ref{T:classification}(e) which is done in
Proposition~\ref{P:orthogonalunderSOn}.  The key step in the proof of this proposition
is to observe that the splitting of $\O_n(\kk)$-orbits of $\sl_2$-triples
$(e,h,f)$ in $\so_n(\kk)$ with $e,f \in \cN$  into
$\SO_n(\kk)$-orbits lines up exactly with the splitting of $\O_n(\kk)$-orbits in $\cN$ into
$\SO_n(\kk)$-orbits.

\subsection*{Acknowledgments}
The results in this paper form part of the PhD research of the second author, who is grateful to the EPSRC
for financial support.  The first author is supported by EPSRC grant EP/R018952/1.  We thank Alexander Premet, David Stewart and
Adam Thomas for some helpful comments, and the referee for some helpful suggestions

\section{Preliminaries}

Throughout the rest of this paper $\kk$ is an algebraically closed field of characteristic $p>2$.
All algebraic groups and Lie algebras we work with are over $\kk$.
The prime subfield of $\kk$ is denoted by $\F_p$.
 We use the notation $\s = \sl_2(\kk)$.

\subsection{\texorpdfstring{$\sl_2$}{sl\_2}-triples and some representation theory of \texorpdfstring{$\sl_2(\kk)$}{sl\_2(k)}} \label{ss:sl2modules}

We recall the definition of an $\sl_2$-triple in a Lie algebra $\g$.  We say that
$(e,h,f)$ is an {\em $\sl_2$-triple} in $\g$ if $e,h,f \in \g$ and $[h,e]=2e$, $[h,f] = -2f$ and
$[e,f] = h$.  In other words an $\sl_2$-triple in $\g$ is the image of the standard
ordered basis $(e,h,f)$ of $\s = \sl_2(\kk)$ under a homomorphism $\s \to \g$.  We note that
there is a conflict of notation in that we are using
$e$, $h$ and $f$ to denote both elements of the standard basis of $\s$, and also elements of $\g$.  This
abuse of notation does not tend to cause confusion, so we allow it
in this paper.

We give some notation for $\s$-modules and related $\sl_2$-triples that is
used in this paper.
Given an $\s$-module $M$ and $x \in \s$, we write
$x_M \in \gl(M)$ to denote linear transformation given by the action of $x$ on $M$.
Then we have that $(e_M,h_M,f_M)$ is an $\sl_2$-triple in $\gl(M)$, and in fact
lies in $\sl(M)$ as $\s$ is equal to its derived subalgebra and $\sl(M)$
is the derived subalgebra of $\gl(M)$.

We next recall some aspects of the representation theory of $\s$ that we require later.
This involves explaining the classification of simple $\s$-modules
on which $e$ and $f$ act nilpotently, and some information about extensions between
these simple modules.
The reader is referred to \cite[Section 5]{JantzenLA} for more details, and
an explanation of the classification of all simple $\s$-modules, though
our notation is a bit different.

As above we write $\{e,h,f\}$ for the standard basis of $\s$.  Then the
universal enveloping algebra $U(\s)$ of $\s$ has PBW basis
$\{e^a h^b f^c \mid a,b,c \in \Z_{\ge 0}\}$.  The {\em $p$-centre} of
$U(\s)$ is the subalgebra of the centre of $U(\s)$ generated by $\{e^p,h^p-h,f^p\}$,
and is in fact the polynomial algebra generated by these elements.

Given any simple $U(\s)$-module $M$ the linear transformations $e_M^p$, $h_M^p-h_M$ and $f_M^p$
must act as scalars by Quillen's Lemma.  Thus if $e_M$ and $f_M$ act nilpotently, then $e_M^p = 0$ and $f_M^p = 0$.
Therefore, the simple modules for $U(\s)$ on which $e$ and
$f$ act nilpotently are precisely the modules for
$U_{\star,0}(\s) := U(\s)/\langle e^p,f^p \rangle$.  In fact there
exists $d \in \kk$ such that $M$ is a module for the reduced enveloping
algebra $U_{d,0}(\s):= U(\s)/\langle h^p-h-d^p,e^p,f^p \rangle$.
Since $h^p-h$, $e^p$ and $f^p$
lie in the centre of $U(\s)$ we see that $U_{\star,0}(\s)$ has basis
$\{e^a h^b f^c \mid a,b,c \in \Z_{\ge 0},\ a,c < p\}$ and $U_{d,0}(\s)$
has basis $\{e^a h^b f^c \mid a,b,c \in \Z_{\ge 0},\ a,b,c < p\}$.

Let $\b := \kk h \oplus \kk e$, which is a Borel subalgebra of $\s$.  Define
$U_{\star,0}(\b) := U(\b)/\langle e^p \rangle$.  For $d \in \kk$,
the 1-dimensional $U_{\star,0}(\b)$-module $\kk 1_d$ is defined
by $e \cdot 1_d := 0$ and $h \cdot 1_d := d1_d$.  Then the baby Verma module $Z(d)$ is defined
as $Z(d) := U_{\star,0}(\s) \otimes_{U_{\star,0}(\b)} \kk 1_d$ and has basis
$\{ v_i := f^i \otimes 1_d \mid i \in \{ 0 , \dots, p-1 \} \}$.
The elements of $\s$ act on this basis as
\begin{align*}
h \cdot v_i &= (d - 2i) v_i \\
e \cdot v_i &= \begin{cases}
       \text{$0$} &\quad\text{if $i=0$}\\
       \text{$i(d - i + 1) v_{i-1}$} &\quad\text{if $i>0$}\\
     \end{cases} \\
f \cdot v_i &= \begin{cases}
       \text{$v_{i+1}$} &\quad\text{if $i<p-1$}\\
       \text{$0$} &\quad\text{if $i=p-1$.}\\
     \end{cases}
\end{align*}
Thus, by using that any submodule of $Z(d)$ contains a vector that is killed by $e$,
it can be seen that $Z(d)$ is simple unless $d \in \F_p \setminus \{p-1\}$.  Further,
for $d \in \F_p$, we have that $Z(d)$ has a unique simple quotient $V(d)$ of dimension
$d+1$, where we identify $\F_p = \{0,1,\dots,p-1\} \sub \Z$ to make sense of this dimension.  This determines
all the simple $U_{\star,0}(\s)$-modules on which $e$ and $f$ act nilpotently.

For $c,d \in \{0,1,\dots,p-2\}$, it is known that
\begin{equation} \label{e:extension}
\Ext^1_\s(V(c),V(d)) = 0 \text{ except in the
case } d = p-c-2,
\end{equation}
see for example \cite[Lemma~2.7]{StewartThomas}.
We note that it is well-known that the action of the Casimir element, which
lies in the centre of $U(\s)$, can be used to show that $\Ext^1_\s(V(c),V(d)) = 0$ for $d \ne c,p-c-2$.
Also we note that $\Ext^1_\s(V(d),V(d)) = 0$ can be deduced fairly quickly from
Corollary~\ref{C:Asemisimple}, for $d \in \{0,1,\dots,p-2\}$.

\subsection{Nilpotent orbits for classical groups} \label{ss:nilpotentorbits}
Let $G$ be one of $\GL_n(\kk)$, $\SL_n(\kk)$, $\Sp_n(\kk)$, $\O_n(\kk)$ or $\SO_n(\kk)$ (where we assume $n$ is even
in the $\Sp_n(\kk)$ case), let $\g = \Lie(G)$ and let $\cN$ be the
nilpotent cone of $\g$.
We give an overview of the well-known parametrization of $G$-orbits in $\cN$ in terms of Jordan types, for more
details we refer the reader to \cite[Section 1]{JantzenNO}.

In this paper by a partition we mean a sequence $\lambda =
(\lambda_1, \lambda_2, \dots ,\lambda_m)$
of positive integers $\lambda_i$ such that $\lambda_i \ge \lambda_{i+1}$ for $i = 1,\dots,m-1$; we have
the convention that $\lambda_i = 0$ for $i > m$.
We say that $\lambda$ is partition of $\lambda_1 + \lambda_2 + \dots + \lambda_m$.
We sometimes use superscripts to denote multiplicities in partitions, so
for example may write $(3^2,2,1^3)$ as a shorthand for $(3,3,2,1,1,1)$.
For a partition $\lambda$ and $i \in \Z_{>0}$, we define $m_i(\lambda)$ to be the multiplicity of $i$
in $\lambda$.  Given partitions $\lambda$ and $\mu$ we define $\lambda | \mu$
to be the partition with $m_i(\lambda | \mu) = m_i(\lambda) + m_i(\mu)$
for all $i \in \Z_{>0}$

Let $x \in \cN$.  Then
$x$ has a Jordan normal form, which determines a partition $\lambda$ of $n$.
We refer to $\lambda$ as the {\em Jordan type of $x$}
and write $x \sim \lambda$.  We also use the notation $\lambda(x)$ to denote
the Jordan type of $x$.

It is well-known that in the cases $G=\GL_n(\kk)$ or $G=\SL_n(\kk)$, the $G$-orbits
in $\cN$ are parameterized by their Jordan type.  Also in the
cases $G = \Sp_n(\kk)$ or $G=\O_n(\kk)$ the $G$-orbits in $\cN$ are parameterized by
their Jordan types, and the Jordan types that can occur are known explicitly.
For a partition $\lambda$ of $n$, there is a nilpotent element $x \in \sp_n(\kk)$
(respectively $\so_n(\kk)$) with Jordan type $\lambda$ if and
only if $m_i(\lambda)$ is even for all odd $i$ (respectively $m_i(\lambda)$
is even for all even $i$).
To describe the parametrization in
the case $G = \SO_n(\kk)$, we note that the $\O_n(\kk)$-orbit of $x \in \cN$
is either a single $\SO_n(\kk)$-orbit, or splits into two $\SO_n(\kk)$-orbits.
The former occurs if the centralizer of $e$ in $\O_n(\kk)$ contains
an element of $\O_n(\kk) \setminus \SO_n(\kk)$ whilst the latter occurs
if the centralizer of $e$ in $\O_n(\kk)$ is contained in $\SO_n(\kk)$.
The centralizer of $e$ in $\O_n(\kk)$ is contained in $\SO_n(\kk)$
precisely when all parts of $\lambda$ are even; such partitions
are referred to as {\em very even} as all parts are even and have even multiplicity.

We recall the partial ordering $\preceq$ on the set of
partitions called the {\em dominance order}.  Given partitions
$\lambda$ and $\mu$
we write $ \mu  \preceq \lambda $ if
\[
\sum_{i=1}^r \mu_i \le \sum_{i=1}^r \lambda_i \ \text{for all} \ r \in \Z_{>0}.
\]

We next state a theorem essentially due to Spaltenstein, which shows that the closure
order on nilpotent orbits is determined by the dominance order on partitions.  In the statement
we use the notation $\cO_\lambda$ for the $G$-orbit in $\cN$ of elements with Jordan type $\lambda$.

\begin{Theorem} \label{T:closureorder}
Let $G$ be one of  $\GL_n(\kk)$, $\SL_n(\kk)$, $\Sp_n(\kk)$ or $\O_n(\kk)$.  Let $\lambda$ and
$\mu$ be partitions of $n$ that parameterize a $G$-orbit in $\cN$.
Then $\cO_{\mu} \sub \bar{\cO_{\lambda}}$ if and only if $\mu \preceq \lambda$.
\end{Theorem}

To explain why this theorem holds, we first note that there is a Springer isomorphism
from the variety $\cU$ of unipotent elements in $G$ to $\cN$; that is a $G$-invariant
isomorphism of varieties $\cU \isoto \cN$.
We refer for example to \cite[\S6.20]{Humphreys} for a statement on existence
of Springer isomorphisms and note also that explicit examples of
Springer isomorphisms for $\GL_n(\kk)$, $\SL_n(\kk)$, $\Sp_n(\kk)$ and $\O_n(\kk)$ are given there.
A result of Spaltenstein, \cite[Theoreme~8.2]{Spaltenstein},
establishes that the dominance
order on partitions determines the closure order for the unipotent classes;
we refer also to \cite[Section~13.4]{Carter}, where this result of Spaltenstein is
covered.  Thus Theorem~\ref{T:closureorder} can be deduced using a Springer isomorphism.

We note that the closure order on the nilpotent orbits for the case $G=\SO_n(\kk)$ is also covered
in the result of Spaltenstein.  Here we have that if $\lambda$ and $\mu$
are distinct partitions of $n$ that parameterize $G$-orbits $\cO_{\mu}$ and $\cO_{\lambda}$ in $\cN$,
then $\cO_{\mu} \sub \bar{\cO_{\lambda}}$ if and only if $\mu \preceq \lambda$.
In the case where $\lambda$ is a very even partition, then the two nilpotent
orbits corresponding to $\lambda$ are incomparable.  We note that this can also quickly
be deduced from the $G = \O_n(\kk)$ case with a key step being to note that if $\lambda$ and $\mu$ are both
very even, then there exists some
not very even partition $\kappa$ parameterizing a nilpotent $\O_n(\kk)$-orbit
such that $\mu \preceq \kappa \preceq \lambda$.

\subsection{Standard \texorpdfstring{$\sl_2$}{sl\_2}-triples} \label{ss:standard}

Let $G$ be one of $\GL_n(\kk)$, $\SL_n(\kk)$, $\Sp_n(\kk)$, $\O_n(\kk)$ or $\SO_n(\kk)$ (where we assume $n$ is even
in the $\Sp_n(\kk)$ case), let $\g = \Lie(G)$ and let $\cN$ be the
nilpotent cone of $\g$.  We recall that the $p$-power map on $\g$ is given by
taking the $p$th power of a matrix, so we just write this as $x \mapsto x^p$.

We discuss standard $\sl_2$-triples as introduced by Premet--Stewart in \cite[\S2.4]{PremetStewart}.
This theory of standard $\sl_2$-triples is based on the theory of optimal cocharacters associated
to nilpotent elements developed by Premet in \cite[Section 2]{Premet}.
We note that the material in \cite[\S2.4]{PremetStewart} is
stated only for the case $G$ is a simple group of exceptional type,
and that some of \cite[Section 2]{Premet} works
under the assumption that the derived subgroup of $G$ is simply connected and
there is a non-degenerate $G$-invariant symmetric bilinear from on $G$.
However, the results that we cover go through in our setting,
see for instance the arguments given in \cite[\S2.3]{Premet}.
On the other hand the results we state below can be proved more generally
to cover groups of exceptional type, but this requires some modification
to our arguments.

We recap the construction of standard $\sl_2$-triples given in
\cite[\S2.4]{PremetStewart}.  Let $e \in \cN$ and let $\tau : \kk^\times \to G$ be an optimal cocharacter for $e$,
where by optimal we mean optimal in terms of the Kempf-Rousseau theory as explained in \cite[\S2.2]{Premet}.
By \cite[Theorem 2.3]{Premet}, we may, and do, choose $\tau$ so that $\tau (t)\cdot e = t^2 e$.
Let $h_\tau := \mathrm d\tau(1) \in \g$, then we have $[h_\tau,e] = 2e$.  Let
$G^e$ be the centralizer of $e$ in $G$, let $C_G(\tau)$ be the centralizer of
$\tau$ in $G$, and let $C^e := G^e \cap C_G(\tau)$.  Let $T^e$ be a
maximal torus of $C^e$ and let $L := C_G(T^e)$.  Then $L$ is a Levi subgroup
such that $e$ is distinguished nilpotent in the Lie algebra $\l'$ of the
derived subgroup $L'$ of $L$, that is, the only Levi subalgebra of $\l'$ containing $e$ is $\l'$ itself.
As explained in \cite[\S2.4]{PremetStewart}
there is a unique $f \in \l'$ such that
$(\Ad \tau(t)) f = t^{-2}f$ for all $t \in \kk^\times$, and $[e,f]= h_\tau$.  An $\sl_2$-triple of the
form $(e,h_\tau,f)$ is called a {\em standard $\sl_2$-triple}.

Let $(e,h_\tau,f)$ be a standard $\sl_2$-triple in $\g$.
It is shown in \cite[\S2.4]{PremetStewart}, that we have $f^p = 0$,
and also that if $e^p = 0$, then $e$ is conjugate to $f$ by $G$.
The first of these facts is shown by noting that $f^p$ centralizes
$e$ and $(\Ad \tau(t)) f^p = t^{-2p}f^p$,
but $\tau$ has positive weights on $\g^e$.  The second is proved by using that $\{e,h_\tau,f\}$
spans the Lie algebra of a connected subgroup of $G$ of type $A_1$.

Since $f^{p} = 0$ we can consider $\exp(sf) \in G$ for $s \in \kk$,
see for example the start of the proof of \cite[Proposition~2.7]{PremetStewart}.
Let $\cN(e,h,f) := \{ae+bh+cf \mid a,b,c \in \kk, \ b^2 = -ac\}$
denote the image of the nilpotent cone of $\s = \sl_2(\kk)$ in $\g$.
Standard calculations show that by conjugating $e$ by $\tau(t)$ for $t \in \kk^\times$
and then by $\exp(sf) \in G$ for $s \in \kk$, we obtain that
\begin{equation} \label{e:exp}
\cN(e,h,f) \setminus \kk f =\{t^2(e-sh-s^2f) \mid t \in \kk^\times, s \in \kk\} \sub (\Ad G)e.
\end{equation}
It thus follows that $f \in \bar{(\Ad G)e}$ and so $(\Ad G)f \sub \bar{(\Ad G)e}$.

Now suppose that $e^p \ne 0$.
We can also find a standard $\sl_2$-triple $(f,h_{\tau'},e')$.
Since $f^p=0$, we have that $e'$ is conjugate to $f$, so that
$(e')^p = 0$, and thus $e'$ is not conjugate to $e$ by $G$.
Hence, the $\sl_2$-triples $(f,-h',e')$ and $(f,-h,e)$ are not conjugate by $G$.

Suppose now that $\cV \sub \cN$ is a $G$-stable closed subvariety that contains $e$, and thus
also contains $f$ and $e'$, because $(\Ad G)e' = (\Ad G)f  \sub \bar{(\Ad G)e}$.
Then we see that the map in \eqref{e:Kostantmapthm} is not injective
by considering the $G$-orbits of $(f,-h',e')$ and $(f,-h,e)$, which are distinct,
and both map to the $G$-orbit of $f$.  This argument implies
the following proposition, where in the statement we use the notation
$\cN^p := \{x \in \g \mid x^p = 0\}$.

\begin{Proposition} \label{P:cVsubcNp}
Let $\cV \sub \cN$ be a $G$-stable closed subvariety that satisfies the $\sl_2$-property.
Then $\cV \sub \cN^p$.
\end{Proposition}

\begin{Remark} \label{R:uniquemax}
We explain how Proposition~\ref{P:cVsubcNp} can be used to give a general argument
to prove that there is a unique maximal $G$-stable closed variety $\cV$ of $\cN$
that satisfies the $\sl_2$-property.

Suppose that $\cV$ and $\cV'$ are two such maximal $G$-stable closed subvarieties of $\cN$.  We consider
$\cV \cup \cV'$, which is a $G$-stable closed subvariety of $\cN$, and
it suffices to show that it satisfies the $\sl_2$-property.
Let $(e,h,f)$ and $(e,h',f')$ be $\sl_2$-triples in $\g$ with $e,f,f' \in \cV \cup \cV'$.
Without loss of generality we assume that $e \in \cV$.  By Proposition~\ref{P:cVsubcNp},
we have that $\cV \cup \cV' \sub \cN^p$, and thus $f \in \cN^p$ so that $f^p = 0$.
Thus we can apply the exponentiation argument above
to obtain \eqref{e:exp}, and deduce that $f \in \cV$.  Similarly we can deduce that $f' \in \cV$.  Hence,
as $\cV$ satisfies the $\sl_2$-property, we have that $(e,h,f)$ and $(e,h',f')$
are conjugate by $G$.  Therefore, we have that $\cV \cup \cV'$ satisfies the
$\sl_2$-property as required.
\end{Remark}

\subsection{\texorpdfstring{$\sl_2$}{sl\_2}-triples for \texorpdfstring{$\SL_p(\kk)$}{SL\_p(k)}} \label{ss:SLp}

We consider $\sl_2$-triples for $\SL_p(\kk)$ and recap the known result that for the case $e \sim (p)$
there are multiple $\sl_2$-triples $(e,h,f)$ in $\sl_p(\kk)$ up to conjugacy by $\SL_p(\kk)$.
We present just two non-conjugate such $\sl_2$-triples, but
note that by using the baby Verma module as described in \S\ref{ss:sl2modules}, it
can be shown that there is an infinite family of non-conjugate such $\sl_2$-triples.  In Proposition~\ref{P:SLp}
we explain how this restricts the possible subvarieties $\cV$ of $\cN$ that satisfy the
$\sl_2$-property.  We use the notation from \S\ref{ss:sl2modules} throughout this subsection.

We let $(e_0,h_0,f) := (e_{Z(0)},h_{Z(0)},f_{Z(0)})$ be the
$\sl_2$-triple in $\gl(Z(0))$ determined by the baby Verma module $Z(0)$.  We view
$(e_0,h_0,f)$ as an $\sl_2$-triple in $\sl_p(\kk)$ using the basis of $Z(0)$ given
in \S\ref{ss:sl2modules}.
Similarly there is an $\sl_2$-triple $(e_{p-1},h_{p-1}, f)$ in
$\sl_p$ determined by the baby Verma module $Z(p-1)$ and the basis of $Z(p-1)$ given in \S\ref{ss:sl2modules}.
We note that the $f$ in these $\sl_2$-triples is the same, and that
$e_0 \sim (p-1,1)$ and $e_{p-1} \sim (p)$.  Therefore, the $\sl_2$-triples $(e_0,h_0, f)$
and $(e_{p-1},h_{p-1}, f)$ are not conjugate by $\SL_p(\kk)$, and thus the $\sl_2$-triples
$(f,-h_0, e_0)$ and $(f,-h_{p-1}, e_{p-1})$ are not
conjugate by $\SL_p(\kk)$.  We note here that this implies that $\cV = \cN$ does not satisfy the
$\sl_2$-property for the case $G = \SL_p(\kk)$, as the $\SL_p(\kk)$-orbits of the $\sl_2$-triples
$(f,-h_0, e_0)$ and $(f,-h_{p-1}, e_{p-1})$ are distinct, and map to the same $\SL_p(\kk)$-orbit
under the map in \eqref{e:KostantmapV}.

Let $G$ be one of $\GL_n(\kk)$, $\SL_n(\kk)$, $\Sp_n(\kk)$, $\O_n(\kk)$ or $\SO_n(\kk)$.
Suppose that $G$ has a Levi subgroup $L$ whose derived subgroup $L'$ is
isomorphic to $\SL_p(\kk)$.  By identifying the Lie algebra $\l'$ of $L'$ with $\sl_p(\kk)$,
we may consider the $\sl_2$-triples $(e_0,h_0, f)$
and $(e_{p-1},h_{p-1}, f)$ inside $\g$.  In the cases where $G$ is one of $\GL_n(\kk)$ or $\SL_n(\kk)$,
we have that $e_0 \sim (p-1,1^{n-p+1})$, $e_{p-1} \sim (p,1^{n-p})$ and $f \sim (p,1^{n-p})$; whilst in
the cases $G$ is one of $\Sp_n(\kk)$, $\O_n(\kk)$ or $\SO_n(\kk)$, we have that
$e_0 \sim ((p-1)^2,1^{n-2p+2})$, $e_{p-1} \sim (p^2,1^{n-2p})$ and $f \sim (p^2,1^{n-2p})$.  Thus we see
that $e_0$ is not conjugate to $e_{p-1}$ by $G$, and thus the $\sl_2$-triples $(e_0,h_0, f)$
and $(e_{p-1},h_{p-1}, f)$ are not conjugate by $G$.  Therefore, the $\sl_2$-triples
$(f,-h_0, e_0)$ and $(f,-h_{p-1}, e_{p-1})$ are also not conjugate by $G$.  Hence, if $\cV$ is a $G$-stable
closed subvariety of $\cN$ that contains $f$, then we see that the map in \eqref{e:Kostantmapthm}
is not a bijection, so that $\cV$ does not satisfy the $\sl_2$-property.

As a consequence of the above discussion we obtain the following proposition.

\begin{Proposition} \label{P:SLp}
Let $G$ be one of $\GL_n(\kk)$, $\SL_n(\kk)$, $\Sp_n(\kk)$, $\O_n(\kk)$ or $\SO_n(\kk)$,
and let $\cV$ be a $G$-stable closed subvariety of $\cN$.  Suppose that $\cV$ contains
an element of Jordan type $(p,1^{n-p})$ (respectively $(p^2,1^{n-2p})$) if $G$ is
one of $\GL_n(\kk)$ or $\SL_n(\kk)$ (respectively $\Sp_n(\kk)$, $\O_n(\kk)$ or $\SO_n(\kk)$).
Then $\cV$ does not satisfy the $\sl_2$-property.
\end{Proposition}

We remark that a strengthening of Proposition~\ref{P:SLp} to incorporate groups of
exceptional type is also known, though we choose not to go into that here.

\section{The general and special linear groups} \label{S:glsl}

For the main part of this section we consider the case $G = \GL_n(\kk)$ and work towards proving
Theorem~\ref{T:classification}(a).  Then in \S\ref{ss:sl} we consider the
case $G = \SL_n(\kk)$ and deduce Theorem~\ref{T:classification}(b).

To prove Theorem~\ref{T:classification}(a) we work with the algebra
\[
A:= U(\s)/\langle e^{p-1} , f^{p-1} \rangle,
\]
where we recall that $\s = \sl_2(\kk)$.
In Corollary \ref{C:Asemisimple} we see that $A$ is semisimple, we note that this is given by \cite[Theorem 1]{JacobsonTDSLA},
though we give an alternative proof.

We note that we write elements of $A$ as linear combinations of monomials in $e$, $h$ and $f$,
so there is a possibility of a conflict of notation
with elements of $U(\s)$. However, when considering elements, of $U(\s)$ or $A$,
we ensure it is clear from the context which algebra they are from.

\subsection{Simple \texorpdfstring{$A$}{A}-modules and a lower bound for the dimension of \texorpdfstring{$A$}{A}} \label{ss:lowerbound}

In the following lemma we give a set of pairwise non-isomorphic simple $A$-modules.
The simple $\s$-modules
$V(d)$ in the statement of the lemma are as recalled in \S\ref{ss:sl2modules}.

\begin{Lemma} \label{L:simplemodules}
$V(0)$, $V(1)$, \dots ,$V(p-2)$ are pairwise non-isomorphic simple $A$-modules.
\end{Lemma}

\begin{proof}
For $0 \le d \le p-2$, we have that $e^{d+1}$ and $f^{d+1}$ act as zero on $V(d)$.
Hence $e^{p-1}$ and $f^{p-1}$ act as zero on $V(d)$, thus $V(d)$ is an $A$-module,
and it is simple as an $A$-module as it is simple as an $\s$-module.
For $c \ne d$ we have that $V(c)$ and $V(d)$ have different dimensions, so are certainly
not isomorphic.
\end{proof}

In the following corollary we establish a lower bound for $\dim(A/\rad{A})$, where $\rad{A}$ denotes
the Jacobson radical of $A$.  We achieve this by applying Wedderburn's theorem to the semisimple
algebra $A/\rad{A}$.  For background on the representation theory used here we refer
to \cite[Sections 3 and 5]{CurtisReiner}.

\begin{Corollary} \label{C:lowerbound}
The dimension of $A/\rad{A}$ is greater than
or equal to $\sum_{i=1}^{p-1} i^2$.
\end{Corollary}

\begin{proof}
We have that $A/\rad{A}$ is semisimple and using Lemma~\ref{L:simplemodules} we have that $V(0)$, $V(1)$, \dots $,V(p-2)$  are simple modules for $A/\rad{A}$.
From Wedderburn's theorem we deduce that
\begin{align*}
\dim(A/\rad{A}) \ge \dim(V(0))^2 + \dim(V(1))^2 + \dots + \dim(V(p-2))^2 = \sum_{i=1}^{p-1} i^2. & \qedhere
\end{align*}
\end{proof}

\subsection{A spanning set for \texorpdfstring{$A$}{A} and an upper bound for the dimension of \texorpdfstring{$A$}{A}} \label{ss:spanningset}

We define the subsets
\[
S_k := \{ f^a h^k e^c \mid a, c < p-1-k \}
\]
of $A$ for each $k <p-1$.
The following proposition is proved at the end of this subsection.

\begin{Proposition} \label{P:spanningset}
$S := \bigcup_{k=0}^{p-2} S_k $ is a spanning set for $A$.
\end{Proposition}

We note that $|S_k| = (p-1-k)^2$, thus $|S| = \sum_{i=1}^{p-1} i^2$, which is equal to the
lower bound of $A/\rad A$ given in Corollary~\ref{C:lowerbound}. Therefore, by combining
Corollary~\ref{C:lowerbound} and Proposition~\ref{P:spanningset}, we are able
to deduce that $S$ is a basis of $A$ so $\dim(A) = \sum_{i=1}^{p-1} i^2 $.  Further,
we have that $\rad{A} = 0$, so that $A$ is semisimple.  Hence
we have that $\{V(0), V(1), \dots, V(p-2)\}$ is a complete set
of inequivalent simple $A$-modules.  This is all stated in Corollary~\ref{C:Asemisimple},
and is then used to prove that $\cN^{p-1}$ satisfies the $\sl_2$-property in Corollary~\ref{C:gl}.

In order to show that $S$ is a spanning set for $A$, we start with a lemma
which gives some relations in $U(\s)$.

\begin{Lemma} \label{L:powers}
Within $U(\s)$, for any $k \in \Z_{>0}$, we have
\begin{enumerate}
    \item $[e^k,h] = -2ke^k$,
    \item $[ e^k, f] = khe^{k-1} - k(k-1) e^{k-1}$, and
    \item $[h^k,f] \in \sspan\{ fh^i \mid i = 1,\dots,k-1\}$.
\end{enumerate}
\end{Lemma}

\begin{proof}
(a) We note that $e$ is an eigenvector of $\ad h$ with eigenvalue $2$, so
$e^k$ is an eigenvector of $\ad h$ with eigenvalue $2k$.  Thus $[e^k,h] = -2ke^k$.

\smallskip
\noindent
(b) We use a simple induction on $k$.
For $k=1$, we have $[e,f] =h$. Suppose that $[ e^k, f] = khe^{k-1} - k(k-1) e^{k-1}$.  Then
\[
[e^{k+1},f] = [e^k,f] e + e^k [e,f] = khe^k - k(k-1)e^k +e^kh.
\]
From (a) we have that $e^kh = he^k - 2ke^k$, hence $ [e^{k+1},f]$ is equal to
\[
khe^k -k(k-1)e^k +he^k -2ke^k = (k+1) he^k - (k^2 + k) e^k = (k+1)he^k -k(k+1) e^k.
\]

\smallskip
\noindent
(c) Let $\b_- := \kk f \oplus \kk h$, which is a Borel subalgebra of $\s$.  We have
that $[h^k,f]$ is an element of $U(\b_-)$, and has degree $k$ for the
PBW filtration.  Also let $\tau : \kk^\times \to S = \SL_2(\kk)$ be the
cocharacter such that $\Ad(\tau(t))e = t^2e$, $\Ad(\tau(t))h = h$ and
$\Ad(\tau(t))f = t^{-2}f$.  Then $[h^k,f]$ is an eigenvector of $\ad h$ with eigenvalue $-2$.
Then for the extension of the action of $S$ on $\s$ to $U(\s)$, we see that
$\Ad(\tau(t))[h^k,f] = t^{-2}[h^k,f]$.
Thus by considering a PBW basis of $U(\b_-)$ we must have
that $[h^k,f]$ lies in the span of $\{ fh^i \mid i = 1,\dots,k-1\}$.
\end{proof}

We move on to prove a lemma giving spanning properties of the sets $S_k$.
Before stating and proving this lemma we explain, in the following remark, how we use an
antiautomorphism of $U(\s)$ to reduce the amount of work required.

\begin{Remark} \label{R:equivalencyofeandf}
Consider the antiautomorphism $\sigma : U(\s) \to U(\s) $ determined by
\[
\sigma(e) = f, \  \sigma(h) = h, \ \sigma(f) =e.
\]
So for any $a,b,c \in \Z_{\ge 0}$ we have
\[
\sigma(f^ah^be^c) = \sigma(e)^c\sigma(h)^b\sigma(f)^a = f^ch^be^a.
\]
As $\sigma$ stabilizes $\langle e^{p-1},f^{p-1}\rangle$ it gives an automorphism
of $A$, which we also denote by $\sigma$.
Using $\sigma$ we have that for any relation in $A$ we can find an equivalent
relation where the powers of $e$ and $f$ are swapped.  More precisely, if we have
some $r_{a,b,c} \in \kk$ such that
$\sum_{a,b,c} r_{a,b,c} f^a h^b e^c = 0$, then
\[
0 = \sigma \left(\sum_{a,b,c} r_{a,b,c} f^a h^b e^c \right) =
\sum_{a,b,c} r_{a,b,c} f^c h^b e^a.
\]
Using this, we note that for any relation on elements of $A$
written in the form of a linear combination of monomials $f^a h^b e^c$, there
is a another relation determined by swapping the powers of $e$ and $f$.
We also observe here that $S_k$ is stable under $\sigma$ for all $k$.
\end{Remark}

\begin{Lemma} \label{L:spanning}
Let $k \in \Z_{\ge 0}$.  Within $A= U(\s)/\langle e^{p-1} , f^{p-1} \rangle$, we have
\begin{enumerate}
    \item if $k < p$ and either $a \ge p-1-k$ or $c \ge p-1-k$, then $f^a h^k e^c \in \sspan(S_0 \cup \dots \cup S_{k-1}),$ and
    \item if $k \ge p$, then $f^a h^k e^c \in \sspan(S_0 \cup \dots \cup S_{p-2})$ for any $a,c \geq 0$.
\end{enumerate}
\end{Lemma}

\begin{proof}
We work by induction to show that for $k \in \Z_{\ge 0}$ with $k < p$, if $a \ge p-1-k$ or $c \ge p-1-k$, then
$f^a h^k e^c \in \sspan(S_0 \cup \dots \cup S_{k-1})$.

Note that this is clear for $k=0$, as $e^{p-1}=0 = f^{p-1}$ in $A$.

To demonstrate the argument we also cover the case $k=1$. We just show that $f^ahe^{p-2} \in \sspan(S_0)$ for any
$a<p-1$ as then we have the analogue for $f^{p-2}he^c$ using Remark~\ref{R:equivalencyofeandf}.
We have $e^{p-1}=0$, therefore using Lemma~\ref{L:powers}(b) we see
\[
0 = [ f^ae^{p-1} , f] = f^a[e^{p-1},f]=(p-1) f^ahe^{p-2} - (p-1)(p-2) f^ae^{p-2}
\]
hence we have
\[
f^ahe^{p-2} = -2f^ae^{p-2} \in \sspan(S_0),
\]
and we are done.

Now let $k \in \Z_{\ge 0}$ with $k < p$.
For our inductive hypothesis, we suppose that for
all $i <k$, if $a \ge p-1-i$ or $c \ge p-1-i$, then $f^ah^ie^c \in \sspan(S_0 \cup \dots \cup S_{i-1})$.
We first show that
\begin{equation} \label{e:equation7}
    f^ah^k e^{p-1-k} \in \sspan(S_0 \cup \dots \cup S_{k-1}) \ \text{ for any} \ a<p-1 .
\end{equation}

In order to show this, we first consider some arbitrary $x \in A$ and show that,
if there is some $j \le k$ such that $x \in \sspan(S_0 \cup \dots \cup S_j)$,
then $[x,f] \in \sspan(S_0 \cup \dots \cup S_{j+1})$. It is enough to show
that $[f^ah^je^c,f] \in \sspan(S_0 \cup \dots \cup S_{j+1})$ for any $j \le k$, $a,c< p-1-j$.

Using Lemma~\ref{L:powers}(c)
there exist some $a_i \in \kk$ such that
\begin{align} \label{e:firstequation}
[f^ah^je^c, f] &= f^a ([h^j,f]e^c +h^j[e^c,f]) \nonumber \\
&= f^a \left( \left( \sum_{i=1}^{j-1} a_i fh^i\right)e^c + ch^{j+1}e^{c-1} - c(c-1) h^je^{c-1}\right).
\end{align}
For $i<j\le k$ we have that $f^{a+1}h^ie^c \in \sspan(S_0 \cup \dots \cup S_{j-1})$
using our inductive hypothesis if needed.
As $a, c-1 < p-1-(j+1)$ we have $f^ah^{j+1}e^{c-1}\in S_{j+1}$, and
$f^ah^je^{c-1} \in S_j$. Hence we can conclude that each of the terms
in \eqref{e:firstequation} is in $\sspan(S_0\cup \dots \cup S_{j+1})$ and hence
\begin{equation} \label{e:equation2}
[f^ah^je^c, f] \in \sspan(S_0 \cup \dots \cup S_{j+1}).
\end{equation}

We now move on to prove \eqref{e:equation7}. By our inductive hypothesis, we have
that $f^ah^{k-1}e^{p-k} \in \sspan(S_0 \cup \dots \cup S_{k-2})$, hence
by \eqref{e:equation2} we have that $[f^ah^{k-1}e^{p-k},f] \in \sspan(S_0 \cup \dots \cup S_{k-1})$. Thus
\begin{equation} \label{e:equation4}
[f^ah^{k-1}e^{p-k},f] = f^a[h^{k-1},f]e^{p-k} + f^ah^{k-1}[e^{p-k},f] \in \sspan(S_0 \cup \dots \cup S_{k-1}).
\end{equation}
We show that the first term on the right hand side of \eqref{e:equation4}
is in $ \sspan(S_0 \cup \dots \cup S_{k-2})$.  This is done by noting that if $i < k-1$
then $ f^{a+1}h^ie^{p-k} \in S_i $ and hence in
$\sspan(S_0 \cup \dots \cup S_{k-2})$, and thus, using Lemma~\ref{L:powers}(c) we see that
\[
f^a[ h^{k-1},f]e^{p-k}
\in \sspan(S_0 \cup \dots \cup S_{k-2}).
\]
By rearranging \eqref{e:equation4} we obtain that
$f^ah^{k-1}[e^{p-k},f] \in \sspan(S_0\cup \dots \cup S_{k-1})$,
we then use Lemma~\ref{L:powers}(b) to see
\begin{equation} \label{e:equation5}
f^ah^{k-1}[e^{p-k},f]
= -kf^ah^ke^{p-k-1} - k(k+1)f^ah^{k-1}e^{p-k-1}.
\end{equation}
Note that $f^ah^{k-1}e^{p-k-1}\in S_{k-1}$. Thus, we can rearrange \eqref{e:equation5} to see
\[
kf^ah^ke^{p-k-1}  \in \sspan(S_0 \cup \dots \cup S_{k-1}).
\]
As we have assumed that $0 < k < p$, we have that $k \ne 0$ in $\kk$,
and so we deduce \eqref{e:equation7}

We next show that \eqref{e:equation7} can be used to prove that if $c \ge p-k-1$
we have $f^ah^ke^c \in \sspan(S_0 \cup \dots \cup S_{k-1})$. We know by \eqref{e:equation7}
that we can find some scalars $r_{i,b,j} \in \kk$ so that
\[
f^ah^ke^{p-1-k} = \sum_{\substack{0 \le b \le k-1 \\ i,j < p-1 -b}} r_{i,b,j} f^ih^be^j.
\]
We consider $f^ah^ke^l$ for $l>p-1-k$, and have that
\[
f^ah^ke^l = \sum_{\substack{0 \le b \le k-1 \\ i,j < p-1 -b}} r_{i,b,j} f^i h^b e^{j+(l-p-1-k)}.
\]
Using the induction hypothesis $f^ih^be^{j+(l-p-1-k)} \in \sspan(S_0 \cup \dots \cup S_b)$,
and hence $f^ah^ke^l \in \sspan(S_0 \cup \dots \cup S_{k-1})$.

We can use the antiautomorphism $\sigma$ from Remark~\ref{R:equivalencyofeandf}
to show that $f^ah^ke^c \in \sspan(S_0 \cup \dots \cup S_{k-1})$ when $a \ge p-1-k$, and so we have
completed the proof of (a).

In fact we have proved that for any $k \in \Z_{\ge 0}$ with $k <p$ and any $a,b \in \Z_{\ge 0}$ that
we have $f^a h^k e^b \in \sspan(S_0 \cup \dots \cup S_{p-2})$.  As a particular case, we
have that $h^{p-1} \in  \sspan(S_0 \cup \dots \cup S_{p-2})$.  Now given  $f^a h^k e^b$ with $k \ge p$,
we can repeatedly substitute $h^{p-1}$ as an expression in $\sspan(S_0 \cup \dots \cup S_{p-2})$ and
obtain $f^a h^k e^b$ as a linear combination of terms $f^ih^be^j$ with $b \le p-1$.  From this we
can deduce that $f^a h^k e^b \in \sspan(S_0 \cup \dots \cup S_{p-2})$ using part (a) of the lemma.
Thus we have proved part (b) of the lemma.
\end{proof}

Using Lemma~\ref{L:spanning} we are now able to show that $S$ is a spanning set for $A$,
and hence prove Proposition~\ref{P:spanningset}.

\begin{proof}[Proof of Proposition~\ref{P:spanningset}]
We have that $\{f^ah^be^c \mid a,b,c \in \Z_{\ge 0} \}$ is a basis for $U(\s)$,
hence as $A$ is obtained from $U(\s)$ by taking the quotient by $\langle e^{p-1} , f^{p-1} \rangle$ we see that
\[
\{ f^ah^be^c \mid a,c < p-1,\ b \in \Z_{\ge 0} \}
\]
spans $A$. By Lemma~\ref{L:spanning}, every element in this set is
contained in the span of $S$.  Hence, $S$ is a spanning set for $A$.
\end{proof}

\subsection{Proof of Theorem~\ref{T:classification}(a)}

Let $G= \GL_n(\kk)$.
We recall that $\cN^{p-1}$ is defined in \eqref{e:Np-1}.  In Corollary~\ref{C:gl} it
is stated that $\cN^{p-1}$ has the $\sl_2$-property.  To prove this corollary we use the fact that
$A$ is semisimple.  The semisimplicity of $A$ is stated as part of the following corollary, which is
proved as explained after the statement of Proposition~\ref{P:spanningset}.

\begin{Corollary} \label{C:Asemisimple}
We have that $S$ is a basis of $A$, so the dimension of $A$ is equal to
$\sum_{i=1}^{p-1} i^2$. Further, we have that $A$ is
semisimple, and the simple modules of $A$ are $V(0)$, $V(1)$, \dots, $V(p-2)$.
\end{Corollary}

\begin{Remark}
We note that further results can be proved using the arguments for the proof
of Corollary~\ref{C:Asemisimple} (or deduced from its statement).  For any $m < p$,
it can be shown that $U(\s)/\langle e^m, f^m \rangle$ is a semisimple
algebra with simple modules $V(0)$, $V(1)$, \dots, $V(m-1)$; also this
statement can be proved for the case of $\s = \sl_2(\C)$ for any $m \in \Z_{>0}$.
These results are also covered in \cite[Theorem 2]{JacobsonTDSLA}.
\end{Remark}

We now explain how $A$-modules relate to $\sl_2$-triples in $\g$.
Any $A$-module $M$ can be considered as an
$\s$-module, and thus we obtain an $\sl_2$-triple $(e_M, h_M, f_M)$ in $\gl(M)$,
as explained in \S\ref{ss:sl2modules}.  Moreover, we have $e_M^{p-1} = 0 =
f_M^{p-1}$, as $M$ is an $A$-module and $e^{p-1} = 0 = f^{p-1}$ in $A$.
Suppose that $\dim M = n$ and choose an identification $M \iso \kk^n$
as a vector space, then we can view $(e_M, h_M, f_M)$ as an $\sl_2$-triple
in $\g$.  Further, given two $A$-modules $M$ and $N$, both of dimension $n$, we have that $M \iso N$
if and only if the $\sl_2$-triples $(e_M, h_M, f_M)$
and $(e_N, h_N, f_N)$ are conjugate by an element of $G$.

Hence, there is a one-to-one correspondence between the set of $n$-dimensional
$A$-modules up to isomorphism and the $\sl_2$-triples $(e,h,f)$ in $\g$ with
$e,f \in \cN^{p-1}$ up to conjugacy by elements of $G$.
Thus proving that $\cN^{p-1}$ satisfies the $\sl_2$-property
is equivalent to proving that for each partition $\lambda$ of $n$
such that $m_i(\lambda) = 0$ for all $i \ge p$, there is an $n$-dimensional
$A$-module $M_\lambda$ on which $e$ acts with Jordan type $\lambda$, and this module is
unique up to isomorphism.

By Corollary~\ref{C:Asemisimple}, each $A$-module is semisimple and
hence a direct sum of the simple modules $V(0), \dots, V(p-2)$.
So any $n$-dimensional $A$-module satisfies
$M \iso \bigoplus_{d=0}^{p-2} V(d)^{\oplus s_d}$ for some
$s_d \in \Z_{\ge 0}$ with $\sum_{d=0}^{p-2}(d+1)s_d = n$.
We have that $e$ acts on $M$ with Jordan type $\lambda_M$,
where $m_d(\lambda_M) = s_{d-1}$ for each $d$.

Thus we see the desired module is
$M_\lambda := \bigoplus_{d=0}^{p-2} V(d)^{\oplus m_{d+1}(\lambda)}$.
Hence, we have proved the following corollary.

\begin{Corollary}  \label{C:gl}
Let $G = \GL_n(\kk)$.  Then $\cN^{p-1}$ satisfies the $\sl_2$-property.
\end{Corollary}

We now explain that $\cN^{p-1}$ is maximal satisfying the $\sl_2$-property to complete the
proof of Theorem~\ref{T:classification}(a).
Suppose that $\cV$ is a $G$-stable closed subvariety of $\cN$ such that $\cV \not\subseteq \cN^{p-1}$.
Thus there must exist some $e' \in \cV$ with Jordan type $\lambda = ( \lambda_1, \lambda_2, \dots, \lambda_m)$
such that $\lambda_1 \geq p$. Hence, by Theorem
\ref{T:closureorder}, there exists $e \in \cV$ with Jordan type $(p,1^{n-p})$.
Thus, using Proposition \ref{P:SLp}, we deduce that $\cV$ does not satisfy the $\sl_2$-property.

\subsection{Deduction of Theorem~\ref{T:classification}(b)} \label{ss:sl}
In this short subsection we deal with the case $G = \SL_n(\kk)$ and explain
that Theorem~\ref{T:classification}(b) follows quickly from
Theorem~\ref{T:classification}(a).  We let $\bar G = \GL_n(\kk)$ and $\bar \g = \Lie(\bar G)$.

We note that the nilpotent cone $\cN$ in $\g$ is the same as
the nilpotent cone of $\bar \g$, and that two elements in $\cN$ are conjugate
by $\bar G$ if and only they are conjugate by $G$, because
$\bar G$ is generated by $G$ and $Z(\bar G)$.  Also we note
that any $\sl_2$-triple in $\bar \g$ must lie in the derived subalgebra
of $\bar \g$, which is equal to $\g$.  Thus for any $G$-stable
closed subvariety $\cV$ of $\cN$, we have that
the set of $G$-orbits in $\cV$ is equal to the set of $\bar G$-orbits in $\cV$
and the set of $G$-orbits of $\sl_2$-triples $(e,h,f)$ with $e,f \in \cV$ is
equal to the set of $\bar G$-orbits of $\sl_2$-triples $(e,h,f)$ with $e,f \in \cV$.
It is now clear that Theorem~\ref{T:classification}(b) follows from Theorem~\ref{T:classification}(a).

\section{The symplectic, orthogonal and special orthogonal groups} \label{S:resultsforclassicals}

In this section we deal with the cases where $G$ is one of $\Sp_n(\kk)$, $\O_n(\kk)$ or $\SO_n(\kk)$ and prove
parts (c), (d) and (e) of Theorem~\ref{T:classification}.

\subsection{Proof that \texorpdfstring{$\cN^{p-1}$}{N(p-1)} satisfies the \texorpdfstring{$\sl_2$}{sl\_2}-property for
\texorpdfstring{$\Sp_n(\kk)$}{Sp\_n(k)} and \texorpdfstring{$\O_n(\kk)$}{O\_n(k)},
and deduction of Theorem~\ref{T:classification}(c)}
Let $G$ be one of $\Sp_n(\kk)$ or $\O_n(\kk)$.  In Proposition~\ref{P:Np-1} we show that
$\cN^{p-1}$ satisfies the $\sl_2$-property; we recall
that $\cN^{p-1}$ is defined in \eqref{e:Np-1}.  To prove this we want to relate
$G$-conjugacy of $\sl_2$-triples in $\g$ with $\GL_n(\kk)$-conjugacy of $\sl_2$-triples in $\g$,
so that we can apply Theorem~\ref{T:classification}(a).  This link is
given in Lemma~\ref{L:sl2orbitsarethesame} and is based on
\cite[Theorem~1.4]{JantzenNO}, which  tells us that two elements of $\g$ are conjugate
by $G$ if and only if they are conjugate by $\GL_n(\kk)$.  With minor
modifications the proof of \cite[Theorem~1.4]{JantzenNO} goes through to prove the lemma below.

\begin{Lemma} \label{L:sl2orbitsarethesame}
Let $G$ be one of $\Sp_n(\kk)$ or $\O_n(\kk)$, and let $(e,h,f)$ and $(e',h',f')$ be
$\sl_2$-triples in $\g$. Then $(e,h,f)$ and $(e',h',f')$ are in the same $G$-orbit
if and only if they are in the same $\GL_n(\kk)$-orbit.
\end{Lemma}

We move on to prove the main result in this subsection.

\begin{Proposition} \label{P:Np-1}
Let $G$ be one of $\Sp_n(\kk)$ or $\O_n(\kk)$.  Then $\cN^{p-1}$ satisfies the
$\sl_2$-property.
\end{Proposition}

\begin{proof}
By the result of Pommerening in \cite[\S2.1]{Pommerening}, or the theory of
standard $\sl_2$-triples recapped in \S\ref{ss:standard}, we have that the map in
\eqref{e:Kostantmapthm} for $\cV = \cN^{p-1}$ is surjective.

Let $(e,h,f)$, $(e,h',f')$ be $\sl_2$-triples in $\g$ with $e,f,f' \in \cN^{p-1}$.
By Corollary~\ref{C:gl}, these $\sl_2$-triples are conjugate by $\GL_n(\kk)$,
and thus by Lemma~\ref{L:sl2orbitsarethesame} are conjugate by $G$.  This implies
that the map in \eqref{e:Kostantmapthm} for $\cV = \cN^{p-1}$ is injective.
\end{proof}

Now let $G = \Sp_n(\kk)$.  We complete the proof Theorem~\ref{T:classification}(c) by
explaining that $\cN^{p-1}$ is the maximal $G$-stable
closed subvariety of $\cN$ satisfying the $\sl_2$-property.
Let $\cV$ be a $G$-stable closed subvariety of $\cN$ such that
$\cV \nsubseteq \cN^{p-1}$.  Then there is an element in $\cV$ which has Jordan type
$\lambda = (\lambda_1,\lambda_2,\dots,\lambda_m)$, where either $\lambda_1 > p$, or
$\lambda_1 = \lambda_2 = p$.  Using Theorem~\ref{T:closureorder}, we deduce that there is an element in $\cV$
with Jordan type $(p+1,1,\dots,1)$ or $(p,p,1, \dots,1)$.  For the first possibility
we can apply Proposition~\ref{P:cVsubcNp} to deduce that $\cV$ does not satisfy
the $\sl_2$-property, whilst in the second case we can apply Proposition~\ref{P:SLp}
to deduce that $\cV$ does not satisfy the $\sl_2$-property.

\subsection{Proof of Theorem~\ref{T:classification}(d)} \label{ss:orthogonal}

Let $G = \O_n(\kk)$ and recall that ${}^1\cN^p$ is defined in \eqref{e:1Np}.
In Proposition~\ref{P:orthogonal}, we prove that ${}^1\cN^p$ satisfies
the $\sl_2$-property.  This proof requires some analysis of underlying $\s$-modules,
where we recall that $\s = \sl_2(\kk)$,
and we note that the ideas in the proof have some similarities with those in the proof
of \cite[Lemma~6.2]{StewartThomas}.

In the proof of Proposition~\ref{P:orthogonal} we apply some well-known general
results on module extensions, which are stated in Lemma~\ref{L:splitting} for convenience of reference.
We only state this lemma for $\s$-modules, though it is of course applicable more generally.

Before the statement of Lemma~\ref{L:splitting} we introduce some notation.  We use the notation $M \iso A|B$
for $\s$-modules $M$, $A$ and $B$, to mean there is short exact sequence
$0 \to B \to M \to A \to 0$.  When using this notation, we identify $B$ with a fixed
submodule of $M$ and $A$ as the corresponding quotient.  We also use the notation
to cover three (or more) modules, so consider $\s$-modules of the form $A|B|C$,
where $A$, $B$ and $C$ are $\s$-modules, and note there is no need to include brackets
in the notation $A|B|C$.
In part (a) of the statement of Lemma~\ref{L:splitting} we should really define
the module $A|C$ occurring there.  This can be defined as the quotient of
$A|B|C$ by the submodule $B$ given by the splitting $B | C \iso B \oplus C$;
or equivalently as the submodule of $A|B|C$ corresponding to the submodule
$A$ of $A|B$ given by the splitting $A | B \iso A \oplus B$.
The modules $A|B$, $A|C$ and $B|C$ in parts (b) and (c) are defined similarly.

\begin{Lemma} \label{L:splitting}
Let $M$, $A$, $B$, and $C$ be $\s$-modules.
\begin{enumerate}
\item Suppose that $M \iso A|B|C$, and that $A|B \iso A \oplus B$ and $B|C \iso B \oplus C$.
Then $M \iso (A|C) \oplus B$.
\item Suppose that $M \iso (A \oplus B)|C$ and that $B|C \iso B \oplus C$.
Then $M \iso (A|C) \oplus B$.
\item Suppose that $M \iso A|(B \oplus C)$ and that
$A|B \iso A \oplus B$.  Then $M \iso (A|C) \oplus B$.
\end{enumerate}
\end{Lemma}

We note that (a) can be proved by using splitting maps $B \to B|C$
and $B|C \to C$ for the short exact sequence $C \to B|C \to B$ to
construct a short exact sequence $B \to A|B|C \to A|C$.  Then a splitting
map $A|B \to B$ for the short exact sequence $B \to A|B \to A$ can be used
to construct a splitting map $A|B|C \to B$ for the short exact sequence $B \to A|B|C \to A|C$.
We have that (b) and (c) are immediate consequences of (a).

We also require an elementary well-known lemma about the action of
nilpotent elements in $\s$-modules, which is stated in Lemma~\ref{L:part}.
We only state this lemma for $\s$-modules, though it is of course applicable more generally.
In the statement we use the notation given in \S\ref{ss:sl2modules} and
\S\ref{ss:nilpotentorbits}.

\begin{Lemma} \label{L:part}
Let $M$, $A$ and $B$ be $\s$-modules and let $x \in \s$.
Suppose that $M \iso A|B$
and that $x_A$ and $x_B$ are nilpotent.
Then $x_M$ is nilpotent and $\lambda(x_A) | \lambda(x_B) \preceq \lambda(x_M)$.
\end{Lemma}

We give an outline of how this lemma can be proved.  First we identify $M$ and $A \oplus B$ as vector spaces.  We then
note that $x_A + x_B$ is in the closure of the $\GL(M)$-orbit of $x_M$, we see this by observing that $x_A + x_B$
lies in the closure of $\{(\Ad \tau(t)) x_M \mid t \in \kk^\times\}$, where
$\tau : \kk^\times \to \GL(M)$ is the cocharacter such that $\tau(t)a = a$ for all $a \in A$
and $\tau(t)b = tb$ for all $b \in B$. The proof  concludes by noting that $\lambda(x_A+x_B)= \lambda(x_A) | \lambda(x_B)$
and then applying Theorem~\ref{T:closureorder}.

We are now ready to state and prove our main result in this subsection.

\begin{Proposition} \label{P:orthogonal}
Let $G = \O_n(\kk)$.  Then ${}^1\cN^p$ satisfies the $\sl_2$-property.
\end{Proposition}

\begin{proof}
By the result of Pommerening in \cite[\S2.1]{Pommerening}, or the theory of
standard $\sl_2$-triples recapped in \S\ref{ss:standard}, the map in
\eqref{e:Kostantmapthm} for $\cV = {}^1\cN^p$ is surjective.  The rest of the
proof is devoted to proving that this map is in fact injective.

Let $(e,h,f)$ be an $\sl_2$-triple in $\g = \so_n(\kk)$ with $e,f \in {}^1\cN^p$.
Let $V \iso \kk^n$ be the natural module for $G = \O_n(\kk)$, and consider $V$ as a
an $\s$-module by restriction to the subalgebra of $\g$ spanned by $\{e,h,f\}$.
Write $(\cdot\,,\cdot)$ for the
$G$-invariant non-degenerate symmetric bilinear form on $V$.

The idea of the rest of the proof is to determine the structure of the $\s$-module $V$,
and observe that it is determined uniquely up to isomorphism by the Jordan type
of $e$.  Then at the end of the proof we use this and Lemma~\ref{L:sl2orbitsarethesame}
to deduce that the map in \eqref{e:Kostantmapthm} is injective.

Let $M \le V$ be a maximal isotropic $\s$-submodule of $V$,
and consider $M^\perp := \{v \in V \mid (v,m) = 0 \text{ for all } m \in M\} \le V$,
which is an $\s$-submodule of $V$.  As $M$ is isotropic we have the sequence of submodules
\begin{equation} \label{e:submodseries}
0 \le M \le M^\perp \le V.
\end{equation}
We have an $\s$-module homomorphism $\phi : V \to M^*$ defined by
$\phi(v)(m) = (m,v)$, where $M^*$ denotes the dual module of $M$.
This induces an isomorphism $V/M^\perp \iso M^*$, and so
by an abuse of notation we write $M^*$ for $V/M^\perp$.
Also we write $N$ for the $\s$-module $M^\perp/M$, and note that $(\cdot\,,\cdot)$
induces an $\s$-invariant nondegenerate symmetric bilinear form on $N$, which we
also denote by $(\cdot\,,\cdot)$.
Thus the quotients in the sequence in \eqref{e:submodseries} are $M$, $N$ and $M^*$,
or in other words $V \iso M^*|N|M$.

We first consider the $\s$-module $M$.
Suppose that $e_M^{p-1} \ne 0$.  Then $\lambda(e_M)$ contains a part of
size $p$ or greater.  We note that $\lambda(e_M) = \lambda(e_{M^*})$, so that
$\lambda(e_{M^*})$ also contains a part of size $p$ or greater.  Using Lemma~\ref{L:part},
we deduce that $\lambda(e_V)$ must have first and second parts greater than or
equal to $p$, but this is not possible as $e = e_V \in {}^1\cN^p$.  Thus we have
that $e_M^{p-1} = 0$. Similarly, we have $f_M^{p-1} = 0$.

It now follows that from Corollary~\ref{C:Asemisimple} that
we have a direct sum decomposition of the $\s$-module
\[
M = M_1 \oplus \dots \oplus M_r
\]
where each $M_i$ is simple, and $M_i \iso V(d_i)$ for some
$d_i \in \{0,1,\dots,p-2\}$.  We have a corresponding direct sum
decomposition
\[
M^* = M_1^* \oplus \dots \oplus M_r^*
\]
of $M^*$, where $M_i^* \iso V(d_i)$ and is dual to $M_i$ via $(\cdot\,,\cdot)$ for each $i$.

Next we consider the $\s$-module $N$, which we recall has
a nondegenerate symmetric invariant bilinear form.  Let $A$ be a simple submodule of $N$,
and consider $A^\perp \le N$, which is also a submodule of $N$.  Thus
$A \cap A^\perp$ is a submodule of $N$, and as $A$ is simple we have
$A \cap A^\perp$ is equal to $0$ or to $A$.  Suppose that
$A \cap A^\perp = A$, so that $A$ is an isotropic subspace of $N$.
Let $\bar A$ be the submodule of $M^\perp$ corresponding to $A \le M^\perp/M$.
Then $\bar A$ is isotropic and this contradicts that $M$ is a maximal isotropic.
Therefore, $A \cap A^\perp = 0$, so that $A$ is non-degenerate, and thus $N = A \oplus A^\perp$.

Hence, $N$ is a semisimple $\s$-module and in fact we have an orthogonal
direct sum decomposition
\begin{equation} \label{e:Ndecomp}
N = N_1 \oplus \dots \oplus N_s,
\end{equation}
where each $N_i$ is a simple $\s$-module and is a non-degenerate subspace for $(\cdot\,,\cdot)$.

Since $e,f \in {}^1\cN^p$, using Lemma~\ref{L:part}, we have that $e_N^p = 0 = f_N^p$, and that $\lambda(e_N)$ and
$\lambda(f_N)$ have at most one part of size $p$.
It follows that for each $i$ we have $N_i \iso V(c_i)$ for some $c_i \in \{0,1,\dots,p-2\}$
with the possible exception of one $j$ for which $N_j = V(c_j)$ where $c_j \in \kk \setminus \{0,1,\dots,p-2\}$.

We note that for $i$ such that $N_i \iso V(c_i)$ for some $c_i \in \{0,1,\dots,p-2\}$, we must have
that $c_i$ is even, because $e_{N_i} \in \so(N_i)$ and $\lambda(e_{N_i}) = (c_i+1)$,
so $c_i+1$ must be odd as explained in \S\ref{ss:nilpotentorbits}.

If there is a $j$ for which $N_j = V(c_j)$ where $c_j \in \kk \setminus \{0,1,\dots,p-2\}$, then
we can show that we must have $c_j = p-1$.  To see this we consider $h_{N_j} \in \so(N_j)$,
which is a semisimple element of $\so(N_j)$ with eigenvalues $c_j, c_j-2,\dots,c_j-2p+2$.
The eigenvalues of a semisimple element of $\so(N_j)$ must include 0 (and also the multiplicity of
an eigenvalue $a$ must be equal to the multiplicity of the eigenvalue $-a$).  It follows that
we must have $c_j = p-1$.

Next we show that $c_i \ne c_j$ for $i \ne j$.  Suppose that we did have $N_i \iso N_j$ for some
$i \ne j$.  We denote $N_{i,j} = N_i \oplus N_j$ and consider
the $\s$-module $N_{i,j}'= N_i' \oplus N_j'$, where $N_i' = N_i$ and $N_j' = N_j$
as $\s$-modules, but we give $N_i' \oplus N_j'$ a non-degenerate $\s$-invariant symmetric
bilinear form so that $N_i$ and $N_j$ are isotropic spaces that are dual to each other.
We fix an isomorphism $N_{i,j}' \iso N_{i,j}$ as vector spaces with
non-degenerate $\s$-invariant symmetric bilinear forms.
This can be used to view
$x_{N_{i,j}'}$ as an element of $\so(N_{i,j})$ for any $x \in \s$.
By definition we have that $N_i \oplus N_j \iso N_i' \oplus N_j'$ as $\s$-modules, which
implies that the $\sl_2$-triples $(e_{N_{i,j}},h_{N_{i,j}},f_{N_{i,j}})$ and $(e_{N_{i,j}'},h_{N_{i,j}'},f_{N_{i,j}'})$
both viewed inside $\so(N_{i,j})$ are conjugate by $\GL(N_{i,j})$.  Now we can apply Lemma~\ref{L:sl2orbitsarethesame}
to deduce that $(e_{N_{i,j}},h_{N_{i,j}},f_{N_{i,j}})$ and $(e_{N_{i,j}'},h_{N_{i,j}'},f_{N_{i,j}'})$
are conjugate by $\O(N_{i,j})$.  Under the identification $N_{i,j}' \iso N_{i,j} \le N$, we have that
$N_i'$ is an isotropic $\s$-submodule of $N$.  However, then the corresponding submodule
$\bar{N_i'}$ of $M^\perp$ is isotropic, and this contradicts the maximality of $M$ as
an isotropic submodule of $V$.

To summarize our findings about $N$, we have that the orthogonal direct sum decomposition
in \eqref{e:Ndecomp}, satisfies that $N_i \iso V(c_i)$ for some $c_i \in \{0,2,\dots,p-1\}$ for each $i$
and that $c_i \ne c_j$ for $i \ne j$.

Our next goal is to prove that
\begin{equation} \label{e:Mperpdecomp}
M^\perp \iso M \oplus N \iso M_1\oplus \dots \oplus M_r \oplus N_1 \oplus \dots \oplus N_s.
\end{equation}
For $j \in \{1,\dots,r\}$ we define $A_j = \bigoplus_{i \ne j} M_i \le M$.  We
consider $M^\perp/A_j$  and aim to show that
\begin{equation} \label{e:Mperp/Aj}
M^\perp/A_j \iso M_j \oplus N
\end{equation}
Noting that $M^\perp/A_j \iso N|M_j$, we see that by repeated application of Lemma~\ref{L:splitting}(c)
we can deduce \eqref{e:Mperpdecomp} from \eqref{e:Mperp/Aj}.  Thus our aim is to
establish \eqref{e:Mperp/Aj}.

Using \eqref{e:extension} and the fact that the summands in \eqref{e:Ndecomp} are pairwise
non-isomorphic, there is at most one $i$ for which $\Ext_\s^1(M_j,N_i)$ is non-zero.

If $\Ext_\s^1(M_j,N_i) =0$ for all $i$, then we have $N_i|M_j \iso N_i \oplus M_j$ for all $i$, and thus we obtain \eqref{e:Mperp/Aj}
by repeated application Lemma~\ref{L:splitting}(b).

If $\Ext_\s^1(M_j,N_i) \ne 0$ for some $i$, i.e.\ $c_i = p-d_j-2$,
then without loss of generality,
we may assume that $i=1$.  Using Lemma~\ref{L:splitting}(b) we can deduce that
\begin{equation} \label{e:Mperp/Aj2}
M^\perp/A_j \iso (N_1|M_j) \oplus N_2 \oplus \dots \oplus N_s.
\end{equation}
We may assume that $N_1|M_j$ is a non-split extension of $N_1$ by $M_j$ otherwise we obtain
\eqref{e:Mperp/Aj}.  We have that $\dim(N_1|M_j) = p$ and we can
use Corollary~\ref{C:Asemisimple} to say that $e_{N_1|M_j}$ or $f_{N_1|M_j}$ has Jordan type $(p)$.
Without loss of generality we assume that $e_{N_1|M_j}$ has Jordan type $(p)$.
We next consider the $\s$-module $A_j^\perp/A_j$ on which $(\cdot\,,\cdot)$
induces a non-degenerate form.  There an isomorphism $A_j^\perp/M \iso (M^\perp/A_j)^*$
via $(\cdot\,,\cdot)$, and also an isomorphism $N \cong N^*$ as $(\cdot\,,\cdot)$
is non-degenerate on $N$. Thus we have that
\begin{equation} \label{e:Ajperp/M}
A_j^\perp/M \iso (M_j^*|N_1) \oplus N_2 \oplus \dots \oplus N_s.
\end{equation}
Using \eqref{e:Mperp/Aj2} and \eqref{e:Ajperp/M} along with repeated
applications of Lemma~\ref{L:splitting}(b) and (c)
we deduce that
\[
A_j^\perp/A_j \iso (M_j^*|N_1|M_j) \oplus N_2 \oplus \dots \oplus N_s.
\]
From the isomorphism $A_j^\perp/M \iso (M^\perp/A_j)^*$ we obtain an isomorphism
$M_j^*|N_1 \iso (N_1|M_j)^*$.  Thus we deduce that $e_{M_j*|N_1}$ has Jordan type $p$.

Next we consider $e_{M_j^*|N_1|M_j}$.  We can choose a basis for $N_1|M_j$ containing a basis of $M_j$ and such that
the matrix of $e_{N_1|M_j}$ with respect to this basis is a single Jordan block $J_p$ of size $p$; we denote
this matrix by $[e_{N_1|M_j}]$, and use similar notation
for other matrices considered here.  We can pick a basis
of $M_j^*$ such that the matrix $[e_{M_j^*}]$ of $e_{M_j^*}$ is a single Jordan block $J_{d_j+1}$ of size $d_j+1$.
By choosing a lift of the basis of $M_j^*$ to $M_j^*|N_1|M_j$ and combining with the basis of
$N_1|M_j$ we obtain a basis of $M_j^*|N_1|M_j$ for which the matrix of $e_{M_j^*|N_1|M_j}$
has block form
\[
[e_{M_j^*|N_1|M_j}] = \begin{pmatrix}
    J_p & X \\
      & J_{d_j+1}
  \end{pmatrix},
\]
where $X$ is some $p \times (d_j+1)$ matrix.  We can consider the matrix $e_{M_j*|N_1}$
with respect to the basis obtained by projecting our basis of $M_j^*|N_1|M_j$ to $M_j^*|N_1$, and we have
\[
[e_{M_j*|N_1}] = \begin{pmatrix}
    J_{c_1+1} & X' \\
      & J_{d_j+1}
  \end{pmatrix},
\]
where $X'$ consists of the bottom $c_1+1$ rows of $X$.  The Jordan type of $e_{M_j*|N_1}$
is $(p)$, so by considering $[e_{M_j^*|N_1}]$ we see that the bottom left entry of $X'$ must be non-zero.
Thus the bottom left entry of $X$ is non-zero.  By considering $[e_{M_j^*|N_1|M_j}]$, we deduce that the
Jordan type of $e_{M_j^*|N_1|M_j}$ is $(p + d_j +1)$.
Now using Lemma~\ref{L:part}, we deduce that the first part of $\lambda(e_V)$ has size
greater than $p$, which is a contradiction because $e_V = e \in {}^1\cN^p$.
From this contradiction we deduce that $N_1|M_j$ is in fact a split extension, and so
we obtain \eqref{e:Mperp/Aj} as desired.

We have now proved \eqref{e:Mperpdecomp}.  Also note we have an isomorphism $V/M \iso (M^\perp)^*$
via $(\cdot\,,\cdot)$, and an isomorphism $N \iso N^*$ since $(\cdot\,,\cdot)$ is non-degenerate on $N$.
Thus from \eqref{e:Mperpdecomp} we obtain
\[
V/M \iso M^* \oplus N.
\]
Hence, by applying Lemma~\ref{L:splitting}(a) we obtain that
\[
V \iso M^*|M \oplus N.
\]

Our next step is to prove that $M^*|M \iso M^* \oplus M$.  Let us suppose that this is not the case,
then, using Lemma~\ref{L:splitting}(b) and (c) we can find $i$ and $j$ such that the subquotient
$M_i^*|M_j$ of $M^*|M$ is a non-split extension.  Using \eqref{e:extension} and the
fact that $M_i \iso M_i^*$ we have that $i \ne j$.  Without loss of generality we can assume that $i = 1$ and $j = 2$,
and then we have that $d_1 = p-d_2-2$.
We consider the subquotient $M_{1,2} = (M_1^* \oplus M_2^*)|(M_1 \oplus M_2)$ of $M^*|M$.
By using that $\Ext^1_\s(M_1,M_1^*) = 0 = \Ext^1_\s(M_2,M_2^*)$ along with Lemma
\ref{L:splitting}(b) and (c), we obtain that $M_{1,2} \iso (M_1^*|M_2) \oplus (M_2^*|M_1)$.  We have that
$(\cdot\,,\cdot)$ induces a non-degenerate bilinear form on $M_{1,2}$ and that $(M_1^*|M_2)$ and $(M_2^*|M_1)$ are
isotropic subspaces of $M_{1,2}$, which are dual via $(\cdot\,,\cdot)$.  By assumption we have that
$M_1^*|M_2$ is a non-split extension, and it has dimension $p$.  Then by Corollary~\ref{C:Asemisimple}
we have that $e_{M_1^*|M_2}$ or $f_{M_1^*|M_2}$ has Jordan type $(p)$.  Without loss of generality
we assume that $e_{M_1^*|M_2}$ has Jordan type $(p)$.  Since $(M_2^*|M_1)
\iso (M_1^*|M_2)^*$, we also have that $e_{M_2^*|M_1}$ has Jordan type $(p)$.  By using Lemma~\ref{L:part},
we deduce that $\lambda(e_V)$ must have first and second parts greater than or
equal to $p$, but this is not possible as $e = e_V \in {}^1\cN^p$.
This contradiction implies that $M^*|M \iso M^* \oplus M$ as desired.

We have thus far proved that the $\s$-module $V$ is semisimple and has the direct sum
decomposition
\[
V = (M_1^* \oplus \dots \oplus M_r^*) \oplus (N_1 \oplus \dots \oplus N_s) \oplus (M_1 \oplus \dots \oplus M_r)
\]
where $M_i \iso V(d_i) \iso M_i^*$ for each $i$ and $N_j \iso V(c_j)$ for each $j$.
Hence, we see that the isomorphism type of $V$ is uniquely determined by the Jordan type of $e$.

Let $(e,h',f')$ be an $\sl_2$-triple in $\so(N)$ with $f' \in {}^1\cN^p$.  Then writing
$V'$ for the $V$ viewed as an $\s$-module for $\sspan\{e,h',f'\}$, we have that
$V'$ is isomorphic to $V$.  From this we deduce that $(e,h',f')$ is conjugate
to $(e,h,f)$ via $\GL(V) = \GL_n(\kk)$, and thus by Lemma~\ref{L:sl2orbitsarethesame} is conjugate
via $\O(V) = \O_n(\kk)$.  This gives the desired injectivity of the map in \eqref{e:Kostantmapthm}, and
completes this proof.
\end{proof}

All that is left to do to prove Theorem~\ref{T:classification}(d) is to prove that
${}^1\cN^p$ is the unique maximal $G$-stable subvariety of $\cN$ satisfying the $\sl_2$-property.
To show this let $\cV$ be a $G$-stable closed subvariety of $\cN$ such that $\cV \nsubseteq \cN^{p-1}$.
Then there is an element in $\cV$ which has Jordan type
$\lambda = (\lambda_1,\lambda_2,\dots,\lambda_m)$, where either $\lambda_1 > p$, or
$\lambda_1 = \lambda_2 = p$.  Using Theorem~\ref{T:closureorder}, we deduce that there is an element in $\cV$
with Jordan type $(p+2,1,\dots,1)$ or $(p,p,1, \dots,1)$.  For the first possibility
we can apply Proposition~\ref{P:cVsubcNp} to deduce that $\cV$ does not satisfy
the $\sl_2$-property, whilst in the second case we can apply Proposition~\ref{P:SLp}
to deduce that $\cV$ does not satisfy the $\sl_2$-property.

\subsection{Deduction of Theorem~\ref{T:classification}(e)}

We are left to deal with the case $G = \SO_n(\kk)$, and we prove that ${}^1\cN^p$ satisfies the $\sl_2$-property
in Proposition~\ref{P:orthogonalunderSOn}.  This is deduced from Proposition~\ref{P:orthogonal}
along with considerations of how $\O_n(\kk)$-orbits of $\sl_2$-triples $(e,h,f)$ in $\g = \so_n(\kk)$ with $e,f \in \cN$ split
into $\SO_n(\kk)$-orbits.

\begin{Proposition} \label{P:orthogonalunderSOn}
Let $G = \SO_n(\kk)$.  Then ${}^1\cN^p$ satisfies the $\sl_2$-property.
\end{Proposition}

\begin{proof}
We know that the map in \eqref{e:Kostantmapthm} for $\cV = {}^1\cN^p$
is surjective, for the same reasons as the corresponding statement in
Proposition~\ref{P:orthogonal}, that is, by the result of Pommerening in \cite[\S2.1]{Pommerening}, or the theory of
standard $\sl_2$-triples recapped in \S\ref{ss:standard}.

As explained in \S\ref{ss:nilpotentorbits} the $\O_n(\kk)$-orbit of $x \in \so_n(\kk)$
is either a single $\SO_n(\kk)$-orbit, or splits into two $\SO_n(\kk)$-orbits, with the
former case occurring precisely when there exists $g \in \O_n(\kk)$ with $\det g = -1$
such that $gx = xg$.  The underlying argument can also be applied to $\sl_2$-triples in $\so_n(\kk)$.
Thus we have that for an $\sl_2$-triple $(e,h,f)$ in $\so_n(\kk)$ the $\O_n(\kk)$-orbit
of $(e,h,f)$ is either a single $\SO_n(\kk)$-orbit or splits into two $\SO_n(\kk)$-orbits.
Moreover, we have that the $\SO_n(\kk)$-orbit of $(e,h,f)$ is equal to the $\O_n(\kk)$-orbit
if and only if there exists some $g \in \O_n(\kk)$ with
$\det g=-1$, $ge=eg$, $gh=hg$ and $gf=fg$.

Let $\lambda$ be the Jordan type of a nilpotent element in ${}^1\cN^p$.  Using $\lambda$
we construct a specific realization of some $e \in {}^1\cN^p$ with $e \sim \lambda$ and an $\sl_2$-triple
$(e,h,f)$.

Let $V = \kk^n$ be the natural module for $\O_n(\kk)$.  Let $m'_i(\lambda) = \frac{m_i(\lambda)}{2}$
for even $i$.  We can form an orthogonal direct sum decomposition of $V$ of the form
\begin{equation} \label{e:Vdecomp}
V = \bigoplus_{i \text{ odd}} \bigoplus_{j=1}^{m_i(\lambda)} V_{i,j} \oplus \bigoplus_{i \text{ even}} \bigoplus_{j=1}^{m'_i(\lambda)} (U_{i,j} \oplus U'_{i,j}),
\end{equation}
where for odd $i$ each $V_{i,j}$ is a non-degenerate subspace of dimension $i$, and for even $i$ the pair $U_{i,j}$ and $U'_{i,j}$
are isotropic subspaces of dimension $i$, which are in duality under the symmetric bilinear form on $V$.  Corresponding to this decomposition of
$V$ we have a subgroup
\[
H \iso \prod_{i \text{ odd}} \O_i(\kk)^{m_i(\lambda)} \times \prod_{i \text{ even}} \GL_i(\kk)^{m'_i(\lambda)}
\]
of $\O_n(\kk)$.  The Lie algebra of $H$ is
\begin{equation} \label{e:hdecomp}
\h \iso \bigoplus_{i \text{ odd}} \so_i(\kk)^{\oplus m_i(\lambda)} \oplus \bigoplus_{i \text{ even}} \gl_i(\kk)^{\oplus m'_i(\lambda)},
\end{equation}
and is a subalgebra of $\so_n(\kk)$.

We choose $e \in \h$ to be regular nilpotent in each of the summands in \eqref{e:hdecomp}.
Then by construction we see that $e \sim \lambda$.  We can find an $\sl_2$-triple $(e,h,f)$ in $\h$,
for example this now follows from Theorem~\ref{T:classification}(a) and (d).
By Proposition~\ref{P:orthogonal} we know that $(e,h,f)$ lies in the unique $\O_n(\kk)$-orbit
of $\sl_2$-triples in $\so_n(\kk)$ with $e \sim \lambda$.

Suppose that $\lambda$ has an odd part, and let $i \in \Z_{>0}$ be odd and such that $m_i(\lambda) > 0$.
Then we can define $g \in \O_n(\kk)$ by declaring that $g$ acts of $V_{i,1}$ by $-1$ and on all other summands
in \eqref{e:Vdecomp} by 1.  We see that $\det g=-1$, and $g$ lies in the
centre of $H$ so that $ge=eg$, $gh=hg$ and $gf=fg$.  Hence,
the $\SO_n(\kk)$-orbit of $(e,h,f)$ is equal to the $\O_n(\kk)$-orbit, and
hence is the unique $\SO_n(\kk)$-orbit of $\sl_2$-triples with $e \sim \lambda$.

Now suppose that $\lambda$ is very even.  Then we know that the $\O_n(\kk)$-orbit of $e$ splits
into two $\SO_n(\kk)$-orbits, and we let $e' \in \so_n(\kk)$ be a representative of the other
$\SO_n(\kk)$-orbit in $(\Ad \O_n(\kk)) e$.  There is an $\sl_2$-triple $(e',h',f')$, which lies
in the $\O_n(\kk)$-orbit of $(e,h,f)$.  Also $(e',h',f')$ is not in the $\SO_n(\kk)$-orbit of
$(e,h,f)$, as $e$ is not conjugate to $e'$ via $\SO_n(\kk)$.  It follows that the $\O_n(\kk)$-orbit
of $(e,h,f)$ splits into two $\SO_n(\kk)$-orbits, and these are the $\SO_n(\kk)$-orbits of
$(e,h,f)$ and $(e',h',f')$.  Now using Proposition~\ref{P:orthogonal} we deduce
that the $\SO_n(\kk)$-orbit of $(e,h,f)$ is the only orbit mapping to the $\SO_n(\kk)$-orbit
of $e$ by the map in \eqref{e:Kostantmapthm}; and that
the $\SO_n(\kk)$-orbit of $(e',h',f')$ is the only orbit mapping to the $\SO_n(\kk)$-orbit
of $e'$ by the map in \eqref{e:Kostantmapthm}.

We have shown that for each $e \in {}^1\cN^p$, there is a unique $\SO_n(\kk)$-orbit
of $\sl_2$-triples $(e,h,f)$ with $e,f \in {}^1\cN^p$ which maps to the
$\SO_n(\kk)$-orbit of $e$ under the map in \eqref{e:Kostantmapthm}.
This shows that the map in \eqref{e:Kostantmapthm} is injective for
$\cV = {}^1\cN^p$, and hence that ${}^1\cN^p$ satisfies the $\sl_2$-property.
\end{proof}

To complete the proof of Theorem~\ref{T:classification}(e), we are just left to show the maximality
of $\cV = {}^1\cN^p$ subject to satisfying the $\sl_2$-property, but this can be done using
the arguments at the end of \S\ref{ss:orthogonal}.

Hence, we have completed the
proof of all parts of Theorem~\ref{T:classification}.

\end{document}